\documentclass[12pt, final]{l4dc2023}

\title[Universal Approximation with Manifold-Constrained Neural ODEs]{Learning on Manifolds: Universal Approximations Properties using Geometric Controllability Conditions for Neural ODEs}

\usepackage{times}
\usepackage{grffile}
\usepackage{amsmath,amssymb,amsbsy}
\usepackage{hyperref}
\usepackage{color}
\usepackage{comment}
\usepackage{lipsum}
\usepackage[normalem]{ulem}
\usepackage{enumerate}
\usepackage{epstopdf}
\usepackage{bbm}
\usepackage{tikz}

\newcommand{\Rd}{\mathbb{R}^d}
\newcommand{\R}{\mathbb{R}}
\newcommand{\M}{\mathcal{M}}
\newcommand{\U}{\mathcal{U}}
\newcommand{\N}{\mathcal{N}}
\newcommand{\F}{\mathcal{F}}
\newcommand{\A}{\mathcal{A}}

\newtheorem{assumption}[theorem]{Assumption}
\newtheorem{problem}[theorem]{Problem}

\author{\Name{Karthik Elamvazhuthi} \Email{kelamvazhuthi@engr.ucr.edu}\\
 \Name{Xuechen Zhang} \Email{xzhan394@ucr.edu}\\
 \Name{Samet Oymak} \Email{oymak@ece.ucr.edu}\\
 \Name{Fabio Pasqualetti} \Email{fabiopas@engr.ucr.edu}\\
\addr 900 University Avenue, Riverside, CA, USA, 92521}

\begin{document}

\maketitle

\begin{abstract}%
In numerous robotics and mechanical engineering applications, among others, data is often constrained on smooth manifolds due to the presence of rotational degrees of freedom. Common data-driven and learning-based methods such as neural ordinary differential equations (ODEs), however, typically fail to satisfy these manifold constraints and perform poorly for these applications. To address this shortcoming, in this paper we study a class of neural ordinary differential equations that, by design, leave a given manifold invariant, and characterize their properties by leveraging the controllability properties of control affine systems. In particular, using a result due to Agrachev and Caponigro on approximating diffeomorphisms with flows of feedback control systems, we show that any map that can be represented as the flow of a manifold-constrained dynamical system can also be approximated using the flow of manifold-constrained neural ODE, whenever a certain controllability condition is satisfied. Additionally, we show that this universal approximation property holds when the neural ODE has limited width in each layer, thus leveraging the depth of network instead for approximation. We verify our theoretical findings using numerical experiments on PyTorch for the manifolds $S^2$ and the 3-dimensional orthogonal group $SO(3)$, which are model manifolds for mechanical systems such as spacecrafts and satellites. We also compare the performance of the manifold invariant neural ODE with classical neural ODEs that ignore the manifold invariant properties and show the superiority of our approach in terms of accuracy and sample complexity. %
\end{abstract}


\section{Introduction}

Due to complexity of robotics systems, data based modeling has become an important technique for predictive and control tasks in many robotics applications \cite{brunke2022safe}. One important consideration when using data based methods in such applications is the presence of rotational degrees of freedom in mechanical systems \cite{choset2005principles, bullo2019geometric}. These rotational degrees of freedom require that most robotics systems’ states be constrained to  on a lower dimensional surface, or manifold, embedded in a higher dimensional Euclidean space. Therefore, it is essential that any black box model that is used in these applications respect the manifold constraints arising from the physics of the system, without which the model can provide  nonphysical outputs, affecting their applicability. Beyond the concern of nonphysical nature of outputs generated by classical black box models, another reason to take this information of lower dimensionality of the state space into account is that it may also reduce the number of parameters required to fit the model to data, and thus mitigate the curse of dimensionality. 

A class of neural networks that has now become standard in machine learning applications is residual neural networks (ResNets) \cite{he2016deep}. However, ResNets do not necessarily satisfy manifold constraints as desired in robotics applications. For this reason, there has been some effort recently on generalizing ResNets to manifold-valued data. One way to generalize ResNets on manifolds  is by considering 
them as numerical discretizations of parameterized ordinary differential equations (ODEs) called neural ODEs \cite{weinan2017proposal,haber2017stable,chen2018neural}. For example, \cite{falorsi2020neural,lou2020neural} present variations of neural ODEs on manifolds, from which one can derive a ResNet respecting geometric constraints, by appropriately discretizing the neural ODE. This point of view, of ResNets as numerical discretizations of neural ODEs, was originally developed to open tools from dynamical systems and control theory to design and train more stable versions of ResNets. This perspective has also been fruitful in understanding the approximation properties of deep neural networks. For example, in  \cite{tabuada2020universal} investigated the universal approximation properties of deep ResNets using geometric control techniques. In \cite{agrachev2021control}, the authors construct a new class control systems that, similar to neural ODEs, have the capability to approximate diffeomorphisms. 

While generalization of neural ODEs to manifolds have been presented in \cite{falorsi2020neural,lou2020neural}, there has been no work on the approximation capabilities of such neural ODEs, for learning manifold-valued maps. In this paper, we study the approximation property of a class of neural ODEs of manifolds from the point of view of controllability properties of control-affine systems, a popular class of systems studied in the control theory literature. For the purpose  of machine learning   problems, the controllability problem of interest is whether the weights of the control system can be used to control the flow of the ODE in such way that the flow, at  a given time instant, is close to the map that is required to be approximated. This non-traditional nature of control problem makes the question of universal approximation challenging to address, since one is not just required to transfer a given initial condition of the control system to another, but simultaneously transfer a collection of initial conditions to a target collection of final values using the same control inputs. We show that a controllability condition on the vector-fields, well known in the form of {\it bracket generating condition} in geometric control theory literature, suffices to guarantee universal approximation of a sufficiently rich class of diffeomorphisms, including those that can be represented using the flow of a manifold-constrained dynamical system. We also use numerical simulations to verify the approximation properties of the neural ODEs on manifolds and compare their performance with classical neural ODEs that do not take into account the manifold valued nature of the data.

\section{Notation}
In this section, we introduce some notation that will be used throughout the paper. The open ball around $x \in \Rd$ will be denoted by
$B_R(x) := \{r \in \R^d ; |x-y|  < R\}$.  Let $\M \subseteq \R^d$ denote a compact smooth manifold without boundary \cite{lee2013smooth}. 
 We will denote by ${\rm Diff}_0(\M)$ the set of diffeomorphisms on $\M$ that are \textit{isotopic} to the identity map $I$ on $\M$. By \textit{isotopic}, we mean that, $X \in {\rm Diff}_0(\M)$ if there exist a smooth map $Y: [0,1] \times \M \rightarrow \M$ such that $Y(0,\cdot)= I$ and $Y(1,\cdot) = X$. A function $f : \R^d \rightarrow \Rd$ is \textit{locally Lipschitz}, if for each $R>0$, there exists $L_R>0$ such that  $|f(x) - f(y)| \leq L_R |x -y| $ for all $x, y \in B_R(x)$. 
 Given two vector-fields $f: \Rd \rightarrow \Rd$ and $g: \Rd \rightarrow \Rd$, the {\it Lie-Bracket} between the two vector-fields is denoted by $[f,g]$ and given by
 \begin{equation} 
 [f,g]^i(x) = \sum_{j=1}^d f^j(x)\frac{\partial g^i}{\partial x^i}(x) -g^j(x)\frac{\partial f^i}{\partial x^i}(x) 
 \end{equation}
 for all $x \in \Rd$, where $f^i(x)$ refers to the $i^{th}$ element of the vector $f(x)$. Let $\mathcal{V} = \{ f_1, ..., f_m\}$ be a finite number of vector fields. Setting $\mathcal{V}_0 = \mathcal{V}$, for each $i \in \mathbb{Z}_+$, we define in an iterative manner the set of vector-fields $\mathcal{V}^i = \lbrace [g, h];g ,h \in \mathcal{V}^{j},~j =0,..,i-1\rbrace $.   We will denote by ${\rm Lie}_x \mathcal{V}$ the set  of vectors defined by,
 \begin{equation}
 {\rm Lie}_x \mathcal{V} = {\rm span} \{g(x);g \in \mathcal{V}^i, ~ i =0,1,...\}.
 \end{equation}
 We will say that the collection of vector fields $\mathcal{V}$ is \textit{bracket generating} on $\M$ if ${\rm Lie}_x \mathcal{V} = T_x \M$ for all $x \in \mathcal{M}$, where $T_x\M$ is the tangent space of the manifold at $x \in \M$.
\section{Problem Formulation}
In this section, we formulate the main problem being addressed in the paper.
Let $\sigma :\mathbb{R} \rightarrow \mathbb{R}$ be a given {\it activation function}. 
We define $\Sigma : \mathbb{R}^d \rightarrow \mathbb{R}^d$ by
\begin{equation}
\Sigma(x) = [\sigma(x_1),...,\sigma(x_d)]^T
\label{eq:vecact}
\end{equation}

An example of the class of activation functions that we consider is {\it sigmoidal functions} with globally bounded derivatives. An activation function $\sigma$ is said to {\it sigmoidal} if its range lies in $[0,1]$, 
\begin{equation}
\label{eq:clnode}
\lim_{x \rightarrow -\infty} \sigma(x)=0 \ \mathrm{and} \ \lim_{x \rightarrow \infty} \sigma(x)=1.
\end{equation}
One such sigmoidal function is 
\begin{equation}
\label{eq:log}
\sigma(x) = \frac{1}{1+e^{-x}}.
\end{equation}

Let $\M \subseteq \mathbb{R}^d$ be a smooth compact manifold that is known. The learning problem that we consider is the following. Suppose there is an unknown map $\Psi:\M \rightarrow \M$. We wish to learn the map $\Psi$ using samples of input output data $\{(x_1,y_1),...,(x_n,y_n)\}$ where $x_i$ are distributed according to some density function $\rho(x)$. One way construct an approximation is using time-one map $X :\R^d \rightarrow \R^d $of the neural ODE,
\begin{eqnarray}
 \label{eq:node}
\dot{x}(t) = A(t)\Sigma(W(t)x+b(t)) ~ ~~x(0)=x_0
\label{eq:node}
\end{eqnarray}
where $A: [0,T] \rightarrow \mathbb{R}^{d \times d}$, $W: [0,T] \rightarrow  \mathbb{R}^{d \times d}$ and  $b : [0,T] \rightarrow \mathbb{R}^d$ are the weights for the neural network. The time-one map $X$ is defined by setting
 $X(x_0) = x(1)$ for each $x_0 \in \mathbb{R}^d$. The feasibility of approximating $\Psi$ using $X$ has been shown in \cite{tabuada2020universal}, for a large class of activation functions $\sigma$.

A drawback of the previous formulation is that the learning problem does not use prior knowledge of the manifold $\M$. Hence, $x_0 \in \M$ does not necessarily imply that $X(x_0) \in \M$. This is critical in applications such as modeling of mechanical systems where the physical state of the system is known to be confined to a given manifold $\M$. To use this knowledge we propose a generalization of neural ODE inspired by control affine systems in control theory.

To ensure that the output data generated by neural ODE stays on $\M$ we choose some $m$ vector-fields $g_i:\mathbb{R}^d \rightarrow \mathbb{R}^d$ that are smooth and satisfy $g_i(x) \in T_x\M$ for all $x$,  where $T_x\M$ denotes the tangent space of $\M$ at $x$.

 Then we consider the following manifold invariant neural ODE

\begin{eqnarray}
\dot{x}(t) = \sum_{i=1}^m a_i(t)g_i \bigg(x(t ) \bigg)\sigma \bigg(w_i^T(t)x(t)+b_i(t) \bigg ) \\
x(0) = x_0
\label{eq:canode}
\end{eqnarray}
where $ a_i: [0,1] \rightarrow  \R$, $w_i :[0,1] \rightarrow \Rd $ and $b:[0,1] \rightarrow \R$ are the weight parameters of the generalized neural ODE. Let $X: \M \rightarrow \M$ be the flow map generated by  the neural ODE \eqref{eq:canode}, defined by setting $X(x_0) = x(1)$ and all $x_0 \in \M$. 

Given this generalized neural ODE, the approximation problem that we consider is the following.
\begin{problem}
Given $\Psi : \M \rightarrow \M$ and $\epsilon>0$ do there exist weight parameters $(a_i,w_i,b_i)$ such that the time-one map $X$ of the neural ODE on manifold \eqref{eq:canode} satisfies.
\begin{equation}
|X(x) -\Psi (x)| \leq \epsilon
\end{equation}
for all $x \in \M$.
\end{problem}

In order to address the above problem, we will need some additional assumptions. The first assumption that we will make is on the vector fields $ \{g_1,...,g_m\}$

\begin{assumption}
	We make the following assumptions on the vector field $g_i$ 
 \begin{enumerate}
 \item The vector fields $g_i: \Rd \rightarrow \Rd$ are smooth.
  \item The vector fields $g_i$ satisfy $g_i(x) \in T_x\M$ for all $x \in \M$.
 \end{enumerate}
	\label{asmp:vec}
\end{assumption}

Let $\U $ be the subset of functions defined by 
\begin{align*}
\U  = \big \{ u:[0,T] \times  \Rd  \rightarrow \R : u \text{ is piecewise constant in time and } \\ u(t, \cdot) ~ \text{is locally Lipschitz for each} ~t \in [0,1] \big \} 
\end{align*}

In addition to this, we will need some mild assumptions on the activation function $\sigma:\mathbb{R}\rightarrow \mathbb{R}$.
For this purpose, let us define the set of functions  
\begin{align*}
\mathcal{F}=\bigcup_{m\in \mathbb{Z}_+}\{ \sum_{i=1}^m \alpha_i \sigma(w_i^Tx+b_i) \ | \ \alpha_i \in \mathbb{R}, w_i \in \mathbb{R}^d, b_i \in \mathbb{R}\}.
\end{align*}
Note that the set $\mathcal{F}$ is the set of arbitrarily wide single-hidden layer neural networks. Using this set we define the subset of $\U$ that take values over $\F$,
\[ \N  = \big \{ u \in \U: u(t, \cdot)  \in  \F ~\text{for each }~ t \in [0,1] \big \} \]
Similarly, we define the set of functions 
\begin{align*} 
\A  = \big \{ & u \in \U: \text{for some piecewise constant functions} ~ \alpha:[0,1] \rightarrow \R,
 \\
& w:[0,1] \rightarrow \Rd, b:[0,1] \rightarrow \R, \\
&~u(t, \cdot) = \alpha(t) \sigma(w^T(t)x+b(t))    \in  \F ~\text{for each }~ t \in [0,1] \big \} 
\end{align*}

We will make the following assumptions on the activation function $\sigma$.
\begin{assumption}
	\label{asmp:neura}
	We make the following assumptions:
	\begin{enumerate}
		\item \textbf{(Regularity)} The activation function $\sigma$ is globally Lipschitz. That is, there exists $K>0$ such that 
		\begin{equation}
		|\sigma (x) - \sigma (y)| \leq K|x-y|, 
		\label{asmp:neura1}
		\end{equation}
		for all $x,y\in \Rd$.
		\item \textbf{(Density of superpositions)} Given a function $f \in C (\Rd;\R)$, and $R>0$, for every $\delta >0$, there exists a function  $ h \in \mathcal{F}$ such that		
    \[|f(x) -h(x)| +\sum_{i=1}^d |\frac{\partial f}{\partial x_i}(x) -\frac{\partial h}{\partial x_i} (x)|<\delta.\]    
     for (Lebesgue) almost every $x \in B_R(x)$.
  \label{asmp:neura2}
  \end{enumerate}
\end{assumption} 

The above assumption is satisfied by both, the Sigmoid activation function, as well as the Rectified Linear unity (  ReLu) activation function.

\section{Analysis}
Before we begin our analysis we make the standard observation that due to Assumption \ref{asmp:vec}, given $u_i \in \U$ for each $i = 1,...,m$, for the following ODE

\begin{eqnarray}
	\dot{x} =  \sum_{i=1}^m u_i(t,x(t)) g_i \bigg(x(t ) \bigg) \nonumber \\
	x(0) = x_0,
 \label{eq:maincsys}
	\end{eqnarray}

the solution $x(t)$ is in fact global and lies in $\M$. This result follows from  \cite{lee2013smooth}.
\begin{theorem}
Given Assumption \ref{asmp:vec}, and functions $u_i \in \U$, suppose that $x(0) \in \M$, then a unique solution $x:[0,1] \rightarrow \Rd$ to the differential equation \eqref{eq:maincsys}
exists. Moreover, $x(t) \in \M$ for all $ t\in [0,1]$.
\end{theorem}

The general idea behind the main result, stated in Theorem \ref{thm:mainthm}, is the following. There exist a large class of maps that can be approximated using the flow of the ODE \eqref{eq:maincsys} using an appropriate choice of functions $u_i$ for $i = 1,...,m$, as long as the vector-fields $\mathcal{V} = \{ g_1, ..., g_m\}$ satisfy a controllability condition. This result is due to \cite{agrachev2009controllability}. However, given functions $u_i$, it is not in general possible to approximate the vector field of \eqref{eq:maincsys} using the the vector field of the neural ODE \eqref{eq:canode} in any classical sense (for example, in the uniform norm or $L_2$ norm), due to the fact that $\A$ is a much smaller set that $\U$. Despite this, using weak convergence arguments, we can use a sequence of vector-fields of the form of the  the neural ODE \eqref{eq:canode} to {\it weakly converge} to functions $u_i$. This, in turn, results in uniform convergence of the flow of neural ODE \eqref{eq:canode} to that of the differential equation  \eqref{eq:maincsys}, thus allowing us to extend the approximation result of \cite{agrachev2009controllability} to flows of the neural ODE \eqref{eq:canode}.

\begin{theorem}
\label{thm:mainthm}
Given that Assumption \ref{asmp:neura}  holds, suppose that $Y \in {\rm Diff}_0(\M)$ and $\epsilon > 0$.  Additionally, suppose that $\mathcal{V} $ satisfies the Bracket generating property on $\M$. That is, $Lie_x \mathcal{V} = T_x \M$ for all $x \in \M$. Let $X: \M \rightarrow \M$ be the flow map generated by  the neural ODE \eqref{eq:canode}, defined by setting $X(x_0) = x(1)$ and all $x_0 \in \M$. Then there exist piecewise constant parameters $a_i  \in \R$, $w_i:[0,1] \rightarrow  \Rd$ and $b_i:[0,1] \rightarrow \Rd$ such that 
\[\sup_{x \in \Omega} |X(x) - Y(x)| \leq \frac{\epsilon}{2}.\]
for all $x \in \M$.
\end{theorem}
\begin{proof}
Since from $Y \in {\rm Diff}_0(\M)$ and due the bracket generating assumption that $Lie_x \mathcal{V} = T_x \M$, it follows from \cite[Corollary 3.1]{agrachev2009controllability}, that there exists $u \in \U $ such that the corresponding time one map $X$  satisfies,

\[\sup_{x \in \Omega} |X(x) - Y(x)| \leq \frac{\epsilon}{2}.\] The result then follows by Proposition \ref{res:diffapp} by approximating $u$ weakly using an element in $\A$.
\end{proof}

Formally, the above approximation result can be understood as follows. The class of maps isotopic to the identity include flows generated by dynamical systems. Suppose the map $Y$ is the time one map of the ODE,
\begin{equation}
\dot{x}(t) =f(t,x(t))
\end{equation}
for some time-dependent smooth vector field $f :[0,1] 
\times \R^d \rightarrow \Rd$. Additionally, suppose that that $x(0) \in \M$ implies $ x(t) \in \M$ for all $t \in \M$, then $Y$ is isotopic to the identity and hence, can be approximated by the flow of the neural ODE \eqref{eq:canode}. The bracket generating assumption can be interpreted as requiring that the following version of the controllability problem,
\begin{eqnarray}
	\dot{x} =  \sum_{i=1}^m u_i(t) g_i (x(t ) ) \nonumber \\
	x(0) = x_0,~~x(1) =x_1 \nonumber
	\end{eqnarray}
is well-posed for all $x_0, x_1 \in \M$. That is, if for every $x_0,x_1 \in \M$, there exist time-dependent functions $u_i(t)$ such that $x(0) = x_0,~~x(1) =x_1$, then the flow of the  manifold invariant neural ODE \eqref{eq:canode} can be used to approximate any element of  ${\rm Diff}_0(\M)$.

~\\
\section{Numerical Results}
In this section, we numerically test the approximation properties of the neural ODE on manifold, using two examples. For each example, we will compare the  performance of our manifold constrained neural ode with the performance of the classical neural ODE \eqref{eq:clnode}. For this purpose, we discretize the problem in time to solve the problem in Pytorch. For this purpose, we introduce the time discretized version of the problem for the classical neural ODE and then for each case, the manifold invariant neural ODE.

Suppose we are given input-output data $\{(x_0^1,y^1),...,(x_0^p,y^p))\} \subset \mathbb{R}^{d} \times \Rd$ we wish to solve the problem, 
\begin{equation}
\min_{\Theta} \frac{1}{P}\sum_{j=1}^P\|x_M^j-y^j\|_2^2+\frac{\lambda \Delta t}{2} \sum_{N=1}^M\|\Theta_N\|_2^2 
\end{equation}
with $\Theta_N = \{A_N, W_N, b_N\} \in \mathbb{R}^{d \times d} \times \mathbb{R}^{d \times d} \times \mathbb{R}^d$  subject to constraints arising from the forward Euler discretization of \eqref{eq:node}, which is the M-layer deep Resnet,
\begin{equation}
\label{eq:disresnet}
x^j_{N+1}= x^j_{N} + \Delta t A_N\Sigma(W_Nx^j_N+b_N)~~   \text{for}~ N= 0,..., M, ~~ j=1,...,P 
.\end{equation}

Here $\Sigma$ is as defined in \eqref{eq:vecact}, $A_N \in \mathbb{R}^{d \times d}$, $W_N \in \mathbb{R}^{d \times d}$  and $b_N \in \mathbb{R}^{d}$, and $\Delta t = \frac{1}{M}$.

\textbf{Example 1: Supervised Learning on the Sphere} In this example we take the manifold $\mathcal{M} =\mathbb{S}^2 = \{x \in \R^3 ; x^Tx = 1\} \subseteq \mathbb{R}^3$, the two dimensional sphere. We define the matrices,
\begin{eqnarray}
\label{eq:genma}
B_1 = 
\begin{bmatrix}
0 & -1 & 0 \\
1 &  0  & 0 \\
0 & 0  & 0 
\end{bmatrix},
~
B_2 =
\begin{bmatrix}
0 & 0& 1 \\
0 &  0  & 0 \\
-1 & 0  & 0 
\end{bmatrix},
\nonumber \\ 
\end{eqnarray}
Let $g_i(x)  = B_ix$ for all $x \in \mathbb{R}^d$ for $i=1,2$. It can be verified that 
\begin{equation}
B_1 B_2 -B_2 B_1= B_3 =
\begin{bmatrix}
0 & 0& 0 \\
0 &  0  & -1 \\
0 & 1  & 0 
\end{bmatrix}
\end{equation}
and hence that $[g_1,g_2](x) = g_3(x) =B_3x$ for all $x \in \mathbb{S}^2$.  
We consider the following neural ODE on the sphere,
\begin{eqnarray}
\dot{x}(t) =  \sum_{i=1}^2 a_i(t)B_i x(t) \sigma \bigg(w_i^T(t)x(t)+b_i(t) \bigg ) \\
x(0) = x_0
\label{eq:sphnode}
\end{eqnarray}
By construction $g_1(x)=B_1x,g_2(x)= B_2x \in T_x\mathbb{S}^2$ for all $x \in M$.
Define 
\[f_i(t,x,\Theta^i(t)) = a_i(t)\sigma(w_i^T(t)x + b_i(t))\]
for $x \in \mathbb{R}^3$ and $\Theta^i(t) = \{a_i(t),w_i(t),b_i(t)\} \in \R \times \Rd \times \R$, for each $t$.  The discretized problem that we solve is the following,
\begin{equation}
\min_{\Theta} \frac{1}{P}\sum_{j=1}^P\|x_M^j-y^j\|_2^2+\frac{\lambda \Delta T}{2N} \sum_{i=1}^2\sum_{j=1}^N\|\Theta_N\|_2^2 
\end{equation}
subject to the geometric Euler discretization \cite{hairer2006geometric} of the ODE \eqref{eq:sphnode}, a  M-layer manifold invariant Resnet,
\[x^j_{N+1} =e^{\Delta t f_1 (t,x_N^j,\Theta_N^1)B_1 + f_2 (t,x_N^j,\Theta_N^2)B_2}x_N^j \]
for $N= 0,..., M $.
In this example, data is generated using the following ODE,
\begin{eqnarray}
& \dot{x} (t) = x_2(t)g_1 (x(t)) +x_3(t) g_3 (x(t)); \\ \nonumber
& x(0) =x^j_0 \nonumber
\end{eqnarray}
\\
The geometric discretization using matrix exponentials ensures that the solution of the discrete system remains on the manifold for every time step $N$ \cite{hairer2006geometric}. 
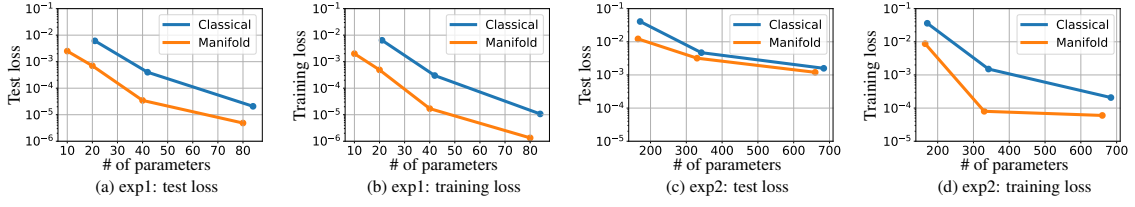
\begin{figure}
\centering
\begin{minipage}[t]{0.23\textwidth}
\begin{tikzpicture}
		\node at (0,0){\includegraphics[width=\linewidth]{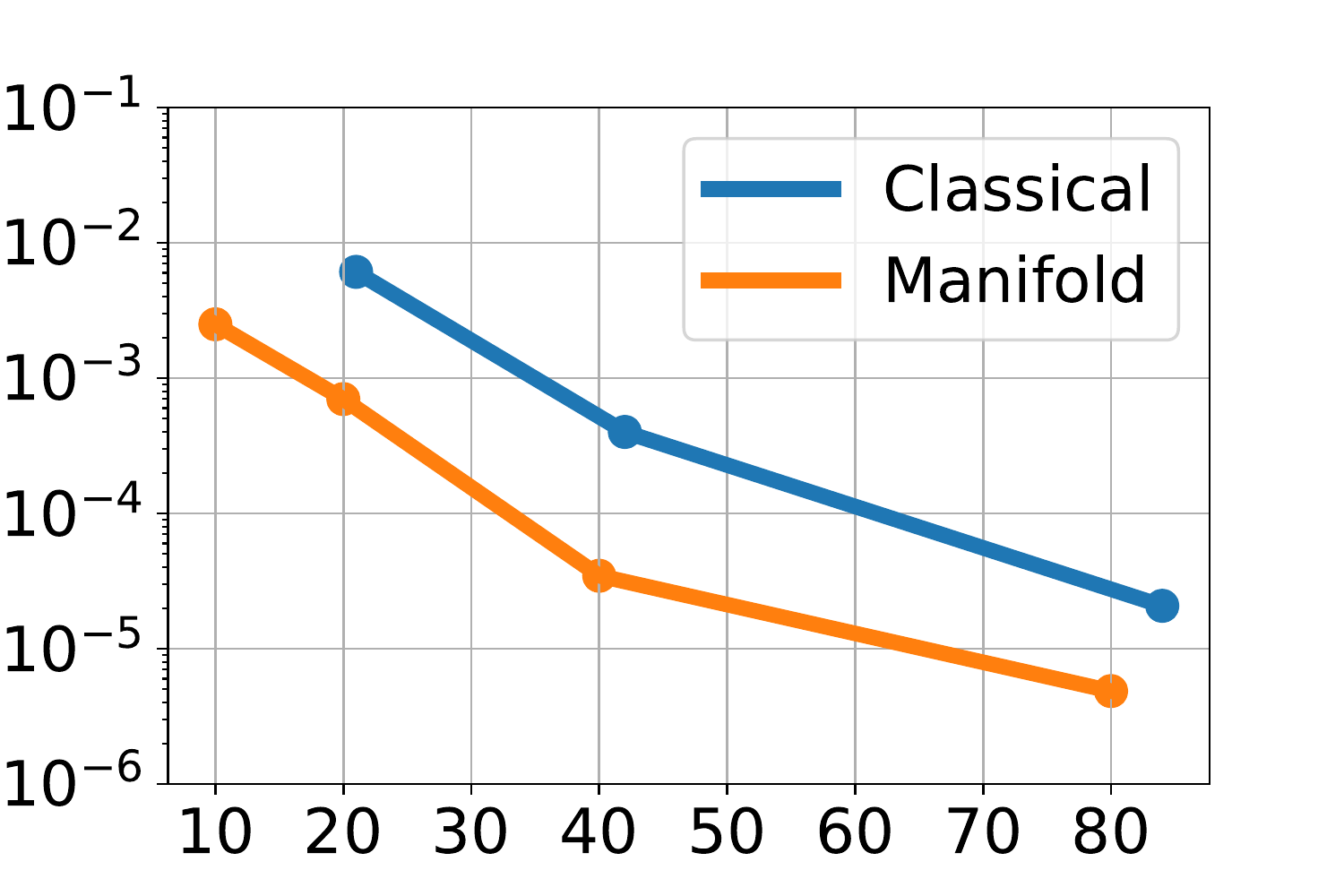}};
		\node at (-1.9 ,0) [scale=0.6,rotate=90] {Test loss};
		\node at (0,-1.2) [scale=0.6] {\# of parameters};
  \node at (0,-1.5) [scale=0.6] {(a) exp1: test loss};
	\end{tikzpicture}
\end{minipage}	
~
\begin{minipage}[t]{0.23\textwidth}
\begin{tikzpicture}
		\node at (0,0){\includegraphics[width=\linewidth]{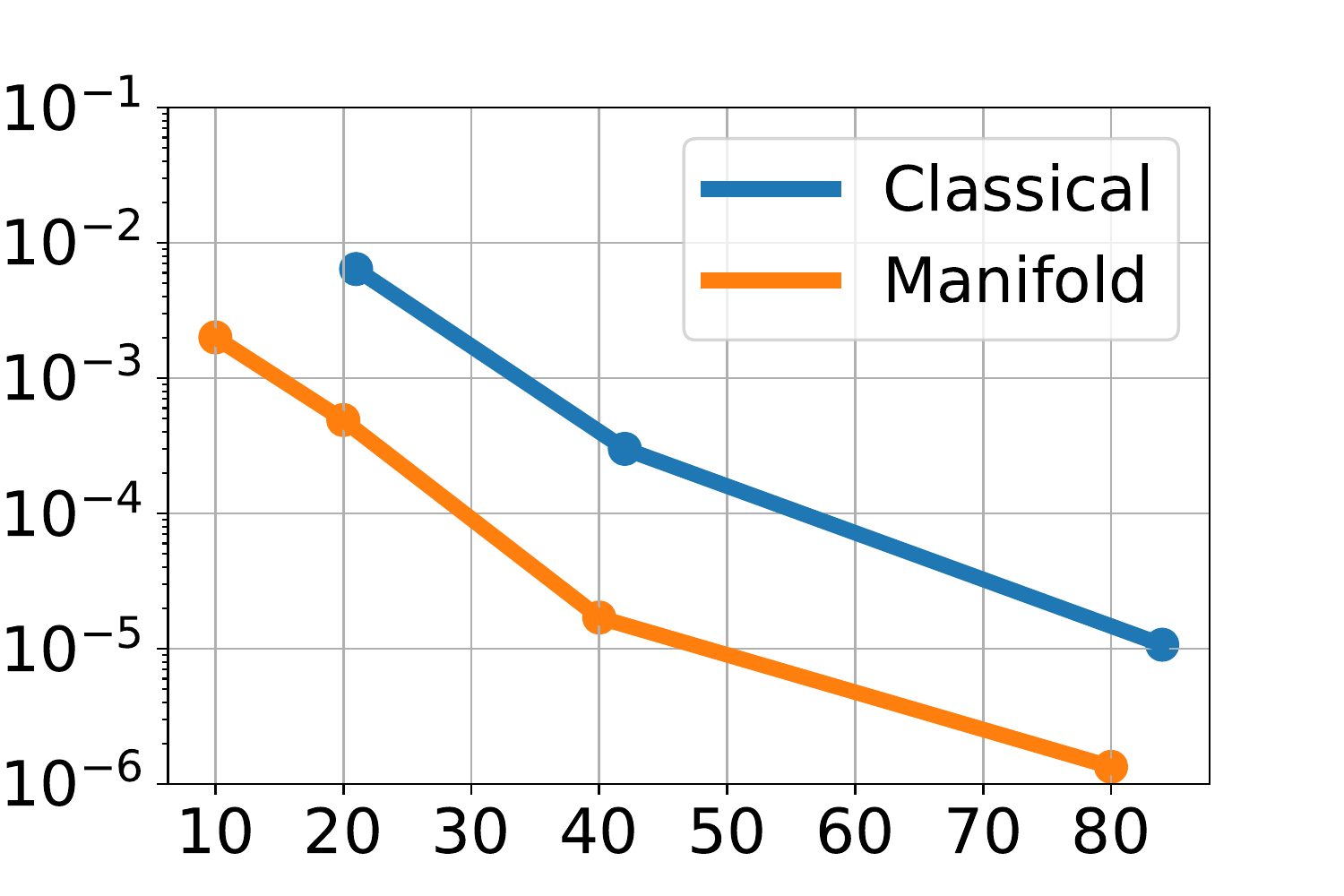}};
		\node at (-1.9 ,0) [scale=0.6,rotate=90] {Training loss};
		\node at (0,-1.2) [scale=0.6] {\# of parameters};
  \node at (0,-1.5) [scale=0.6] {(b) exp1: training loss};
\end{tikzpicture}
\end{minipage}	
~
\begin{minipage}[t]{0.23\textwidth}
\begin{tikzpicture}
		\node at (0,0){\includegraphics[width=\linewidth]{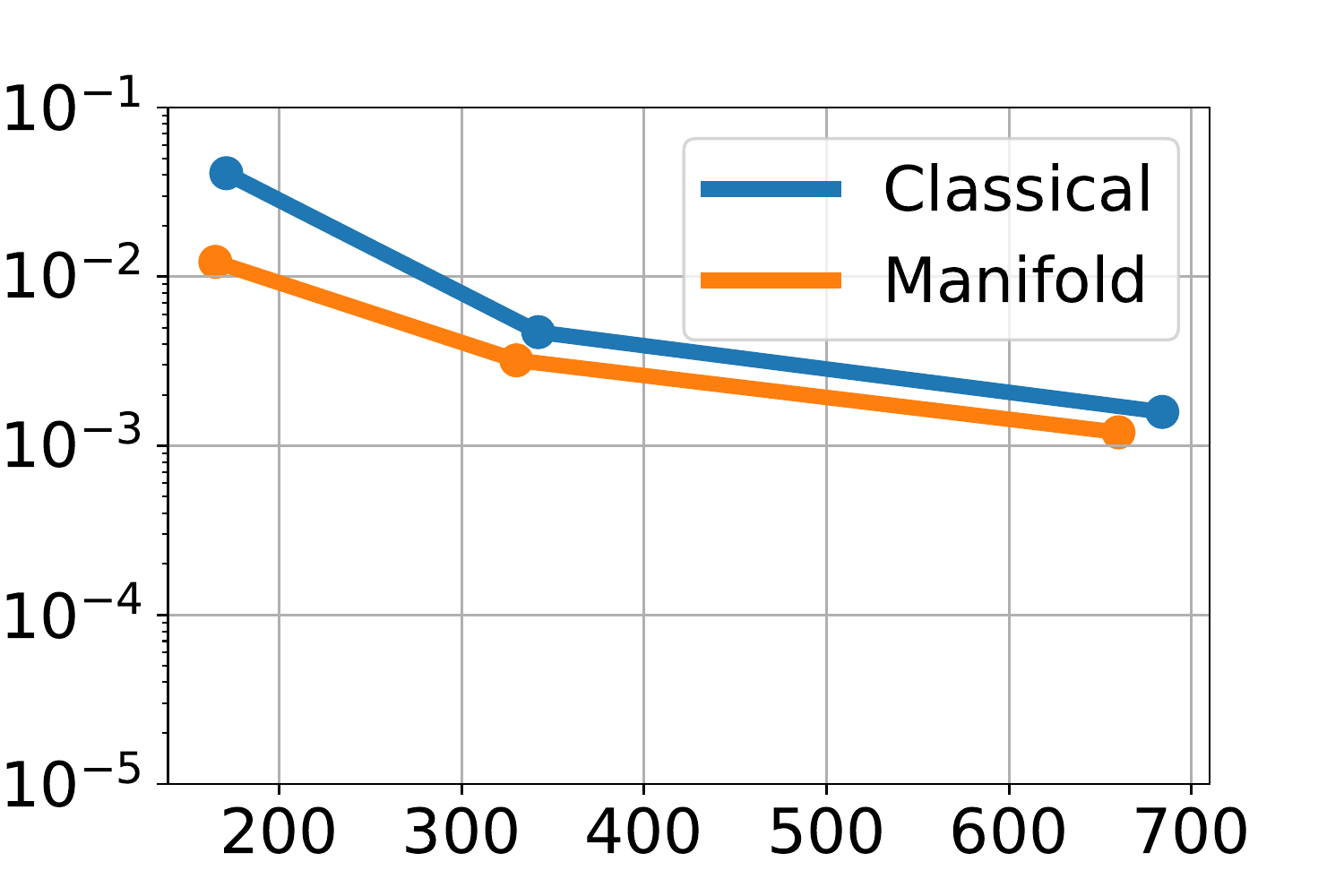}};
		\node at (-1.9 ,0) [scale=0.6,rotate=90] {Test loss};
		\node at (0,-1.2) [scale=0.6] {\# of parameters};
  \node at (0,-1.5) [scale=0.6] {(c) exp2: test loss};
	\end{tikzpicture}
\end{minipage}	
~
\begin{minipage}[t]{0.23\textwidth}
\begin{tikzpicture}
		\node at (0,0){\includegraphics[width=\linewidth]{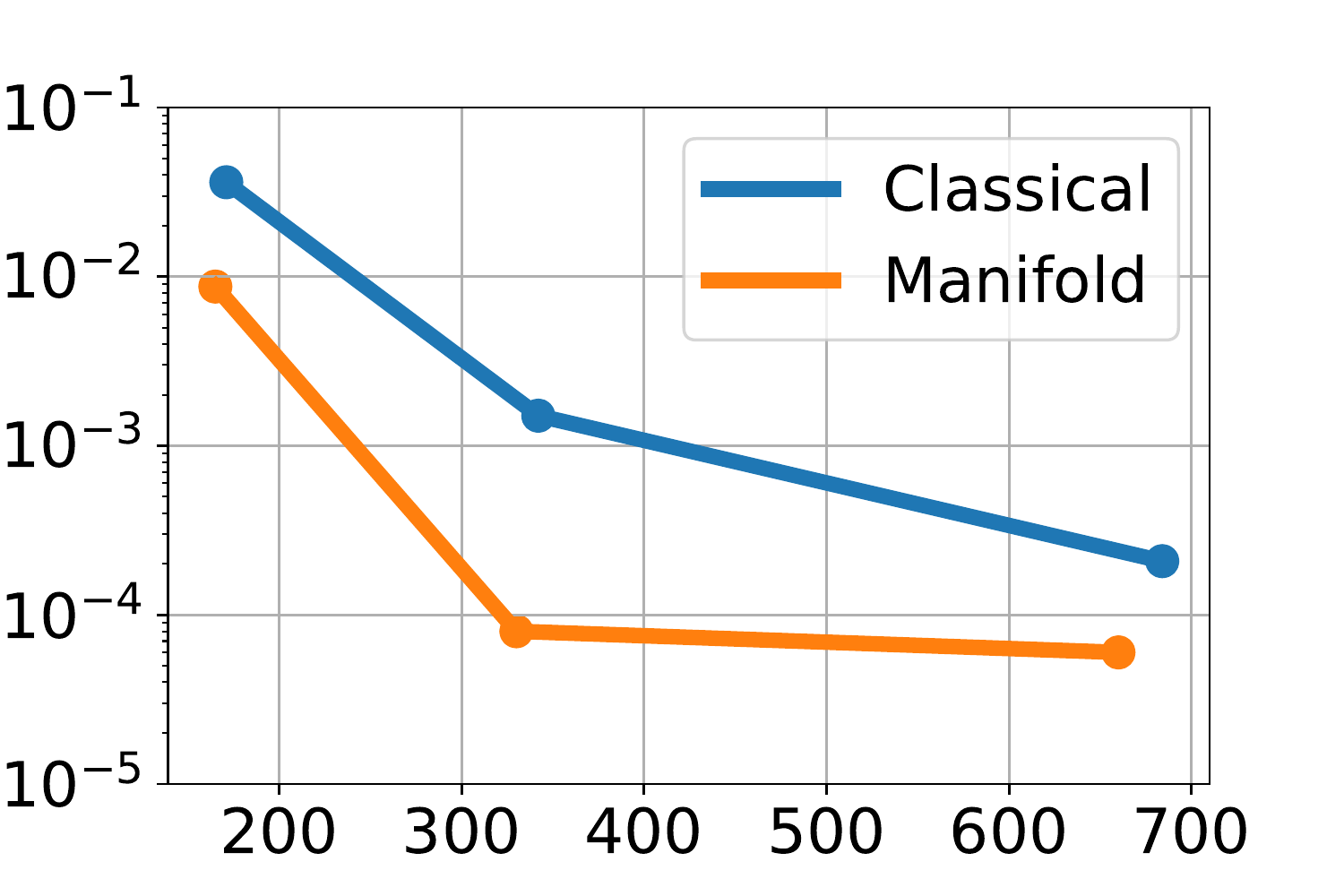}};
		\node at (-1.9 ,0) [scale=0.6,rotate=90] {Training loss};
		\node at (0,-1.2) [scale=0.6] {\# of parameters};
  \node at (0,-1.5) [scale=0.6] {(d) exp2: training loss};
\end{tikzpicture}
\end{minipage}
 \caption{(a) and (b) show the results of example 1. (c) and (d) show the results of example 2.}
	\label{fig:exp1}
\end{figure} 

\textbf{Example 2: Supervised Learning on the 3D rotation group} In this example with consider the 3 dimensional orthogonal groups $\mathcal{M}  = SO(3)  = \{ R \in \R^{3 \times 3} = R^TR = I ,~ \det(R) = 1 \} $ 

\begin{eqnarray}
\label{eq:genma}
B_1 = 
\begin{bmatrix}
0 & -1 & 0 \\
1 &  0  & 0 \\
0 & 0  & 0 
\end{bmatrix},
~
B_2 =
\begin{bmatrix}
0 & 0& 1 \\
0 &  0  & 0 \\
-1 & 0  & 0 
\end{bmatrix},
~
B_3 =
\begin{bmatrix}
0 & 0& 0 \\
0 &  0  & -1 \\
0 & 1  & 0 
\end{bmatrix},
\nonumber \\ 
\end{eqnarray}
In this case we take $g_i(X)  = B_iX$ for all $X \in \mathbb{R}^{3 \times 3}$ for $i=1,2,3$.
We define the neural ODE in the form of a matrix differential equation
\begin{eqnarray}
\dot{X}(t) =  \sum_{i=1}^3 a_i(t)B_i X(t) \sigma \bigg(\langle w_i(t),X(t)\rangle +b_i(t) \bigg ) \\
X(0) = X_0 \in 
\mathbb{R}^{3 \times 3}
\label{eq:sphnode}
\end{eqnarray}
By construction $g_i(X)=B_iX \in T_X{SO(3)}$ for all $X \in M$.
Define 
\[f_i(t,X,\Theta^i(t)) = a_i(t)\sigma(\langle w_i(t),X \rangle + b_i(t))\]
for $X \in \mathbb{R}^{3 \times 3}$ and $\Theta^i(t) = \{a_i(t),w_i(t),b_i(t)\}$ . The discretized problem that we solve is the following,

\begin{equation}
\min_{\Theta} \frac{1}{P}\sum_{j=1}^P\|X_M^j-Y^j\|_2^2+\frac{\lambda \Delta t}{2} \sum_{i=1}^2\sum_{j=1}^M\|\Theta_N\|_2^2 
\end{equation}
subject to M layer manifold Resnet,
\\
\[X^j_{N+1} =e^{\Delta t f_1 (t,X_N^j,\Theta_N^1)B_1+\Delta t f_2 (t,X_N^j,\Theta_N^2)B_2+\Delta t f_3 (t,X_N^j,\Theta_N^3)B_3}X_N^j\]
for $N= 0,..., M $.

In this example, data is generated using the following ODE,
\begin{eqnarray}
& \dot{X} (t) = {\rm Tr}  ( X^2+1)\Big(g_1(X) +g_2(X) +g_3(X (t)) \Big ); \\
& X(0) =X^j_0
\end{eqnarray}
where $\rm{Tr}$ denotes the Trace of the matrix.
\\

\textbf{Implementation}\\
For the experiments, we compare the manifold ODE with classical neural ODE. We use the sigmoid function as the activation function $\sigma$ in all experiments. All the models were trained with standard stochastic gradient descent (SGD) using learning rate 10, and momentum 0.9. The learning rate decays at epochs 500,1000,2000,4000,6000 with a factor 0.8. We started training with a large learning rate to avoid getting stuck into local minima. The classical ODE in all cases needed much more parameters per layer. For Example 1, there are 10 parameters per layer for manifold invariant ODE while there are 21 for classical one. For Example 2, the value is 33 and 171. For a fair comparison, in example 1, we use $M = [1,2,4]$ layers for classical ODE and $M =[1,2,4,8]$ layers for manifold one. In example 2, we use $M =[1,2,4,8]$ for classical ODE and $M =[5,10,20]$ for manifold ODE. 

\textbf{Discussion}\\
The training loss and test loss for the case over the sphere is shown in Figure \eqref{fig:exp1} a-b, respectively. The training loss and test loss for the case over the $SO(3)$ is shown in Figure \eqref{fig:exp1} c-d, respectively. 
Compared to the classical ODE (blue curve), the manifold ODE (orange curve) could always achieve low loss with smaller number of parameters. The result verifies that manifold ODE could truly reduce the number of parameters required to fit the model to data, and mitigate the curse of dimensionality. The greater difference can be observed in the test losses, indicated the superior capability of manifold-invariant neural ODE in terms of generalization.

\section{Conclusion}
We presented a class of neural ODEs for which a given manifold is invariant, and studied their universal approximation properties from the perspective of controllability analysis. Additionally, we use numerical simulations to verify that in terms, of sample complexity and accuracy the neural ODEs on manifolds show superior performance. Possible future directions of work include qualitatively establishing the formal sample complexity estimates of these neural ODEs, and extending them to situations when the manifold $\M$ is unknown.
\appendix

\section{Supplementary results}
For the results stated in this section we will need $L^1(0,1;\Rd)$, the set of vector-valued measurable functions that are integrable over $(0,1)$. We will denote $L^1(0,1;\Rd)$ by  $L^1(0,1)$ when $d=1$. We will say that a sequence $(f^n)_{n=1}^{\infty} $  in  $L^1(0,1;\Rd)$ weakly converges to $f \in L^1(0,1;\Rd)$ if 
 \[ \lim_{n \rightarrow \infty} \int_0^1 \langle \phi(x) , f^n(x) \rangle dx  = \int_{0}^1 \langle  \phi(x) , f(x) \rangle dx  \]
 for all $\phi \in L^\infty(0,1;\R^d)$, where $\langle  \cdot, \cdot \rangle $ denotes the usual inner product on $\Rd$. Given this definition, we have the following result. 
 \begin{lemma}
\label{res:diffapp}
Given Assumption \ref{asmp:neura}, suppose  
that $(u^n_i)_{n=1}^{\infty}$ is a sequence of functions  in $\U$ such that $(u_i^n(x,\cdot))_{n=1}^{\infty}$ is weakly converging to $u_i(x,\cdot)$ in $L^1(0,1)$ for each $x \in \M$ and each $i \in \{1,...,m\}$.  Suppose additionally that the sequence of functions $(u^n_i)_{n=1}^{\infty}$ is bounded and Lipschitz with a common upper bound and common Lipschitz constant on $\M$.
Let $
X^n: \mathcal{M} \rightarrow \mathcal{M}$ be the map given by
   $X^n(x_0) = x^n(1)$ for all $x_0 \in \mathcal{M}$, where $x^n(t)$ is the solution of the differential 
	\begin{eqnarray}
	\dot{x}^n(t) =  \sum_{i=1}^m u^n_i(t,x^n(t)) g_i \bigg(x^n(t ) \bigg) \\
	x^n(0) = x^n_0
	\end{eqnarray}
 Then $(X^n)_{n=1}^{\infty}$ is uniformly converging to $X$ on $\M$.
\end{lemma}
\begin{proof}

Given the density Assumption \ref{asmp:neura}, for every $u^n_i \in \U$, there exists a sequence $(u_i)_{n=1}^{\infty} \in \N$ such that $(u_i(t,\cdot))_{n=1}^{\infty}$ is converging  to $u_i(t,\cdot)$ uniformly on $\M$ for each $i = 1,...,m$, by approximating $(u^{n}_i (t,\cdot))$ using elements in $\N$, that is, using a wide neural network. Moreover, due to Assumption \ref{asmp:neura}, we can assume that the approximations $(u_i)_{n=1}^{\infty}$ have are uniformly bounded almost everywhere on compact subsets of $\Rd$. This implies that the Lipschitz constants of the functions $(u_i)_{n=1}^{\infty} $ are uniformly bounded and have uniformly bounded Lipschitz constants.
Next, from \cite[Lemma IV.7]{elamvazhuthi2022neural}, for each $n \in \mathbb{Z}_+$, $(u^n_i)$ there exists a sequence $(u^{n,m}_i)_{i=1}^{\infty}$ in $\A$ such that  $(u^{n,m}_i(t,\cdot))_{m=1}^{\infty}$ is weakly converging to $u_i^n$ in $L^1(0,1)$, for each $x \in M$, and all $i =1,...,M$. Moreover, from the construction in the statement of \cite[Lemma IV.7]{elamvazhuthi2022neural}, one can see that the sequences $(u^{n,m}_i)_{m=1}^{\infty}$ are uniformly bounded and have a common Lipschitz constant. The result then follows from Theorem \ref{res:diffapp}.
\end{proof}

\begin{proposition}
Given Assumption \ref{asmp:neura}, suppose $u_i \in \U$. Then, for every $\epsilon>0$, there exist  control laws $u_i^{\epsilon}  \in \A$ such that the time one map corresponding to the differential equation 
Let $X^{\epsilon}:\mathbb{R}^n \rightarrow \mathbb{R}^n$ be the map given by $X^{\epsilon}(x_0) =x^{\epsilon}(T)$  where $x^{\epsilon}(t)$ is the solution of the differential equation,
	\begin{eqnarray}
	\dot{x} =  \sum_{i=1}^m u^\epsilon_i(t,x^{\epsilon}(t)) g_i \bigg(x^{\epsilon}(t ) \bigg) \\
	x(0) = x_0
 	\end{eqnarray}
for each initial condition $x_0\in \mathbb{R}^n$ and solution $x(t)$ of the differential equation \eqref{eq:maincsys} satisfies 
\begin{equation}
\sup_{x \in K} |X(x) -X^{\epsilon}(x)| \leq \epsilon 
\end{equation}
\end{proposition}

\begin{proof} 
\textbf{of Theorem \ref{thm:mainthm}} \\
Consider the sequence of vector fields $(v^n)_{n = 1}^ \infty$ defined by

\[v(t,x) =  \sum_{i=1}^m u^n_i(t,x^n) g_i \big(x^n \big)\]
for all $x \in \Rd$. Since $(u^n_i)_{n=1}^{\infty}$ is bounded and Lipschitz with a common upper bound and common Lipschitz constant on $M$, it follows that   $(v^n)_{n = 1}^ \infty$ are bounded and Lipschitz with a common upper bound and common Lipschitz constant on $\M$, the same properties hold true for   $(v^n)_{n = 1}^ \infty$ on $\M$.
It follows from \cite[Lemma 2.7]{pogodaev2016optimal} that $(X^n(x_0))_{n=1}^{\infty}$ converges to $X(x_0)$ for each $x_0 \in \M$. Note that the Lemma is stated for $x_0 \in \R^d$, when the boundedness and Lipschitz constant assumption hold on $\Rd$, but going through the proof, one can infer that the result is also true when $x(t) \in \M$ for all $ t\in [0,1]$ and $(u^n_i)_{n=1}^{\infty}$ is bounded and Lipschitz with a common upper bound and common Lipschitz constant on $M$. 

Next, as the following standard arguments shows, it follows that the maps $(X^n)_{n=1}^{\infty}$ are uniformly Lipschitz. Let $x_0, y_0 \in \M$. Consider the inequality,
\begin{align*}
|x(t)-y(t)| \leq |x_0 - y_0| + \int_0^t |v^n(t,x(t))- v^n(t,y(t))|dt \\
\leq |x_0 - y_0| + L \int_0^t |x(t)- y(t)|dt \\
\end{align*}
From Gronwall's inequality, it follows that
\begin{align*}
|x(t)-y(t)| \leq  e^{Lt} |x_0 - y_0| \\
\end{align*}
for all $t \in [0,1]$. Hence, $|X^n(x_0) -X^n(y_0) |\leq |x_0 - y_0| e^{L}$ for all $x_0,y_0 \in \M$. and hence the sequence $(X^n)_{n=1}^{\infty}$ are uniformly Lipschitz continuous. This it implies that the sequence is equicontinuous. From the Arzelà–Ascoli theorem, we can conclude that there exists a subsequence of $(X^n)_{n=1}^{\infty}$ that is uniformly converging to a map $\tilde{X}$ on $\M$. But we know that $(X(x_0))_{n=1}^{\infty}$ is converging to $(X(x_0))$ for all $x_0 \in \M$. Therefore, the convergence must be uniform.
\end{proof}

\bibliography{l4dc}

\end{document}